\documentclass[letterpaper, 10 pt, conference]{ieeeconf}  

\IEEEoverridecommandlockouts                              

\usepackage{afterpage}
\usepackage{epsfig} 
\usepackage{times} 
\usepackage{amsmath} 
\usepackage{amssymb}  
\usepackage{amsfonts}
\usepackage{mathptmx} 
\usepackage{tabularx}
\usepackage{epstopdf}
\usepackage{subfigure}
\usepackage{float}
\usepackage{epic,color}
\usepackage{mathrsfs}
\usepackage{multirow} 
\usepackage[T1]{fontenc}
\newcounter{rmnum}

\newcounter{anum}

\newlength{\noteWidth}
\long\def\notes#1{\ifinner
             {\tiny #1}
             \else
              \marginpar{\parbox[t]{\noteWidth}{\raggedright\tiny #1}}
               \fi}

\def\IEEEQEDclosed{\mbox{\rule[0pt]{1.3ex}{1.3ex}}}

\def\qed{\ifmmode\IEEEQEDclosed\else{\unskip\nobreak\hfil
\penalty50\hskip1em\null\nobreak\hfil\IEEEQEDclosed
\parfillskip=0pt\finalhyphendemerits=0\endgraf}\fi}

\def\qed{\hspace*{\fill}~\IEEEQED\par\endtrivlist\unskip}

\def\ddt{\frac{\ud}{\ud t}}
\def\Re{\mathbb{R}}

\def\ddt{\frac{\ud}{\ud t}}

\def\notes#1{\marginpar{\tiny #1}\typeout{Notes!
Notes!
Notes!
}}
\renewcommand{\notes}[1]{\typeout{notes!}}
\def\FRAC#1#2#3{\genfrac{}{}{}{#1}{#2}{#3}}
\def\half{{\mathchoice{\FRAC{1}{1}{2}}%
{\FRAC{2}{1}{2}}%
{\FRAC{3}{1}{2}}%
{\FRAC{4}{1}{2}}}}

\def\Re{\field{R}}
\def\k{\sf K}

\def\clZ{{\cal Z}}

\def\E{{\sf E}}

\def\IEEEQEDclosed{\mbox{\rule[0pt]{1.3ex}{1.3ex}}}
\def\qed{\nobreak\hfill\IEEEQEDclosed}

\newcommand{\expect}{ {\sf E} }

\def\clZ{{\cal Z}}

\newtheorem{theorem}{Theorem}
\newtheorem{lemma}{Lemma}
\newtheorem{remark}{Remark}
\newtheorem{proposition}{Proposition}

\def\beq{\begin{eqnarray}} 
\def\bc{\begin{center}} 
\def\be{\begin{enumerate}}
\def\bi{\begin{itemize}} 
\def\bs{\begin{small}}
\def\bS{\begin{slide}}
\def\ec{\end{center}} 
\def\ee{\end{enumerate}}
\def\ei{\end{itemize}}
\def\es{\end{small}}
\def\eS{\end{slide}}
\def\eeq{\end{eqnarray}}

\newcommand{\newP}[1]{\medskip\noindent{\bf #1:}}

\newcommand{\trace}{\text{Tr}}

\newcommand{\PP}{\sf P}

\newcommand{\ud}{\,\mathrm{d}}

\def\Re{\mathbb{R}}
\def\E{{\sf E}}

\def\clZ{{\cal Z}}

\renewcommand{\Re}{\mathbb{R}}

\def\FRAC#1#2#3{\genfrac{}{}{}{#1}{#2}{#3}}

\def\ddt{\frac{\ud}{\ud t}}

\newcommand{\ii}{k}

\newcommand{\Xii}{S}

\newcommand{\hatXii}{\hat{S}}

\newcommand{\Sigmaii}{\tilde{\Sigma}}
\def\kii{\tilde{\sf K}}
\newcommand{\var}{\text{Var}}
\newcommand{\X}{X}
\newcommand{\Om}{\Omega}
\newcommand{\NN}{\mathcal{N}}

\title{\LARGE \bf
An Optimal Transport Formulation of the Linear Feedback Particle Filter}

\author{Amirhossein Taghvaei, Prashant G. Mehta 
\thanks{Financial support from the NSF CMMI grants 1334987 and 1462773 is gratefully acknowledged. 
}
\thanks{A.~Taghvaei and P.~G.~Mehta are with the Coordinated
  Science Laboratory and the Department of Mechanical Science and
  Engineering at the University of Illinois at Urbana-Champaign (UIUC).
{\tt\scriptsize taghvae2@illinois.edu; mehtapg@illinois.edu;}}
}

\begin{document}

\maketitle
\thispagestyle{empty}
\pagestyle{empty}

\begin{abstract}
Feedback particle filter (FPF) is an algorithm to numerically approximate the solution of the nonlinear filtering problem in continuous time. 
The algorithm implements a feedback control law for a system of particles such that the empirical distribution of particles approximates the posterior distribution. 
However, it has been noted in the literature that the feedback control law is not unique. 
To find a unique control law, the filtering task is formulated here as an optimal transportation problem between the prior and the posterior distributions. 
Based on this formulation, a time stepping optimization procedure is proposed for the optimal control design.
A key difference between the optimal control law and the one in the original FPF, is the replacement of noise term with a deterministic term. 
This difference serves to decreases the simulation variance, as illustrated with a simple numerical example.

\end{abstract}
\section{INTRODUCTION}
Nonlinear filtering is concerned with the problem of computing the posterior distribution of a hidden Markov process $\X_t$, given a time history of observations $Z_t$. 
In the linear Gaussian setting, the solution is given by the Kalman-Bucy filter.
For the general nonlinear non-Gaussian case, the problem is infinite dimensional and only approximate numerical solutions are possible.

The feedback particle filter (FPF) algorithm provides one such approximate numerical solution~\cite{taoyang_TAC12},~\cite{yanlaumehmey12}. 
In the FPF algorithm, the posterior distribution is approximated by empirical distribution of an ensemble of particles. 
The particles are realizations of a stochastic process, denoted in this paper as $\Xii_t$. 
The $\Xii_t$ process is simulated according to, 
\begin{equation}
\ud \Xii_t = u_t(\Xii_t)\ud t + \k_t(\Xii_t) \ud Z_t,
\label{eq:u-k}
\end{equation}
where $u_t$ and $\k_t$ are control laws, designed such that $\Xii_t$ is distributed according to posterior distribution of $X_t$, i.e,
\begin{equation*}
\Xii_t \sim \PP(X_t | \clZ_t),\quad \forall t\geq 0,
\label{eq:consistency} 
\end{equation*}
where $\clZ_t:=\sigma\{Z_s;0 \leq s \leq t\}$ is the filtration generated by observation. 
When this condition is true, the filter is said to be exact. 

In the original FPF algorithm, a particular choice of $u_t$ and $\k_t$ is proposed such that the filter is exact. 
Nevertheless, in its general form \eqref{eq:u-k}, there are infinitely many choices of $u_t$ and $\k_t$ that all lead to exactness. 
Investigating the optimality properties of the choice made in FPF and selecting a unique optimal control law is the focus of the current work. 

The optimality concept considered here is motivated by the optimal transportation literature \cite{villani2003}, \cite{evans}.  
One can interpret feedback particle filter as transporting the initial distribution at time $t=0$ (prior) to the conditional distribution at time $t$ (posterior). 
Clearly, there are infinitely many maps that transport one distribution into another. 
Selecting a unique and optimal map is the basic problem of optimal transportation and also informs the investigation of this paper (see Figure~\ref{fig:opt-transp}).

\begin{figure}[t]
\centering
\includegraphics[width=0.8\columnwidth]{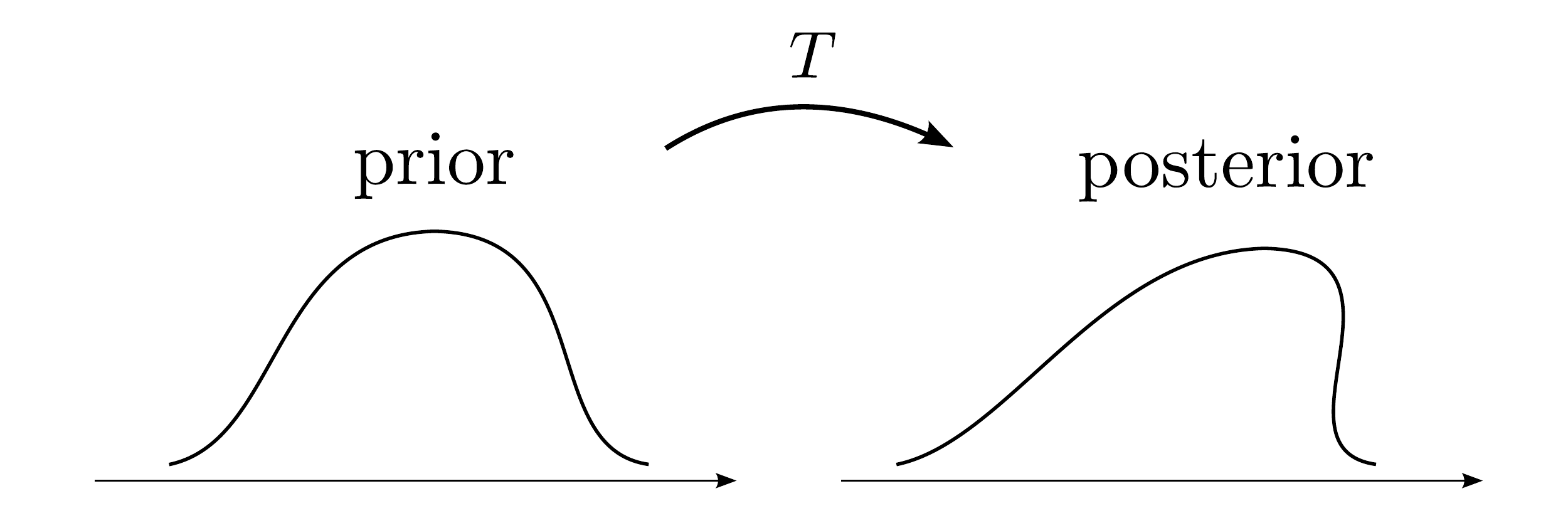}
\caption{Filtering viewed as transporting prior distribution to posterior distribution.}
\label{fig:opt-transp}
\end{figure}

Based on the concept of optimal transportation, a time-stepping optimization procedure is proposed here to obtain a unique optimal control law, denoted as $u^*_t $ and $\k^*_t$. 
In this procedure, a finite time interval is divided into discrete time steps $\{t_0,t_1,\ldots,t_n\}$. 
Then a discrete time random process, $\Xii_{t_k}$ is constructed by initializing $\Xii_{t_0}$ according to the initial prior $\PP(X_0)$, and sequentially evolving $\Xii_{t_k} \to \Xii_{t_{k+1}}$ at each time-step with a map denoted by $T_k$:
\begin{equation*}
\Xii_{t_{k+1}}=T_k(\Xii_{t_k}),\quad \Xii_0 \sim \PP(X_0).
\end{equation*}
The map $T_k$ is obtained by solving an optimal transportation problem between the conditional probability distributions at time instants $t_k$ and $t_{k+1}$, i.e between   $\PP(X_{t_k}|\clZ_{t_k})$ and $\PP(X_{t_{k+1}}|\clZ_{t_{k+1}})$. 
By construction, $\Xii_{t_k}$ is distributed according to $\PP(X_{t_k}|\clZ_{t_k})$ for all $k \in \{1,\ldots,n\}$. 
Finally, by taking the continuous-time limit, a continuous time process $\Xii_t$ and therefore the unique control law, $u^*$ and $\k^*$ for \eqref{eq:u-k} are obtained. 
Figure~\ref{fig:time-step} illustrates the procedure.

The contributions of this paper are as follows:
An optimal transport formulation of the feedback particle filter is introduced. 
The new formulation helps to solve the uniqueness issue for the FPF.
For the linear Gaussian filtering problem, the optimal transport FPF is obtained and compared to the original FPF. 
Using a simple example, the advantage of optimal formulation in reducing simulation variance is analytically and numerically demonstrated.

There are two streams of literature that are relevant to this work: 
(i)~Particle filter algorithms that seek to replace the importance sampling step with a deterministic approach~\cite{daum10},~\cite{mitter03},~\cite{crisan10},~\cite{taoyang_TAC12}; 
(ii)~Optimal transportation techniques for uncertainty propagation that includes synthesis of optimal transport maps for implementing the Bayes rules as a  special case~\cite{reich13}, \cite{reich11},~\cite{MarzoukBayesian}. 
The present work integrates these two streams, that furthermore highlights the optimal transportation roots of the FPF algorithm. 

The evaluation of the FPF algorithm and its comparison to other nonlinear filtering algorithms appears in~\cite{berntorp2015},~\cite{stano2014}~\cite{tilton2012}. 
For additional applications of optimal transportation see~\cite{TannenbaumTexture},~\cite{ColemanPosterior},~\cite{parno14},~\cite{benamou2000}.


The outline of the remainder of this paper is as follows:
In section \ref{sec:consistency-nonuniqueness} the linear FPF is introduced and the uniqueness issue is discussed. Section \ref{sec:opt-transp} includes a short background on the optimal transportation problem.
In Section \ref{sec:opt-transp-sde}, the time stepping optimization procedure is proposed and applied to the linear Gaussian filtering problem. 
Finally Section \ref{sec:numerics} includes a comparison between the original FPF and the optimal transport FPF.

\begin{figure}[t]
\centering
\includegraphics[width=0.95\columnwidth]{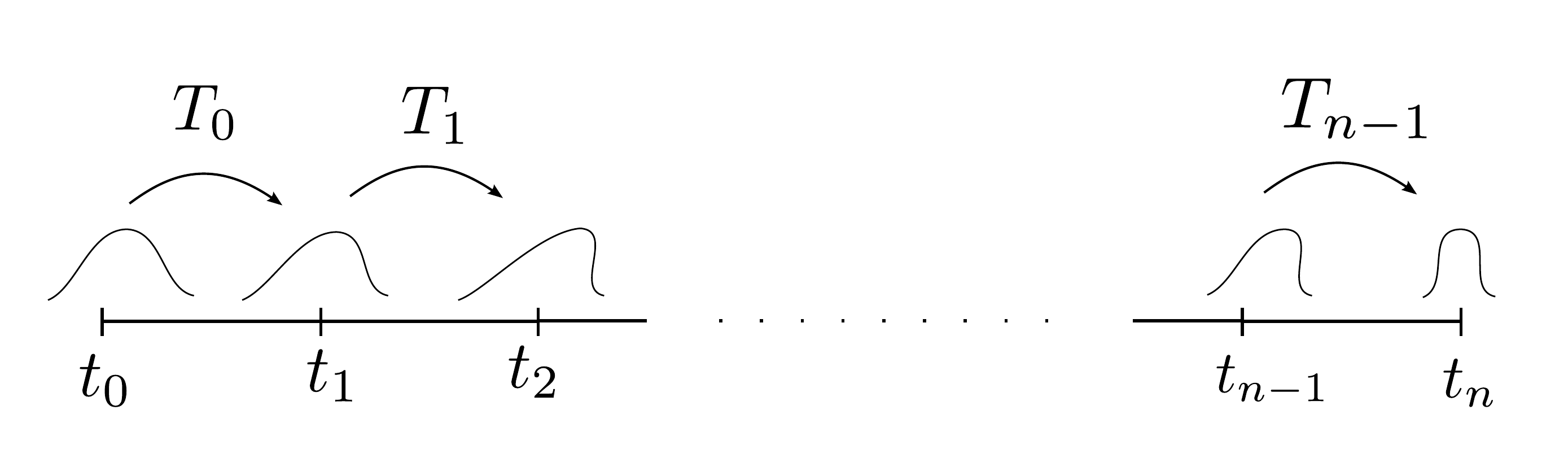}
\caption{The time stepping optimization procedure.}
\label{fig:time-step}
\end{figure}

\newP{Notation}
$\PP$ denotes a probability distribution. $X \sim \PP$ means the random variable $X$ is distributed according to $\PP$.
$\NN(\mu,\Sigma)$ is a Gaussian probability distribution with mean $\mu$ and covariance $\Sigma$. $A \succ 0$ means A is a positive definite matrix.
All sdes are to be interpreted in It\^o sense.

\section{Exactness and Non-uniqueness}
\label{sec:consistency-nonuniqueness}
\newP{Linear Gaussian filtering} Consider the linear Gaussian filtering problem:
\begin{subequations}
\begin{align}
\ud \X_t &= A \X_t \ud t + \ud B_t, 
\label{eqn:Signal_Process}
\\
\ud Z_t &= C \X_t\ud t + \ud W_t,
\label{eqn:Obs_Process}
\end{align}
\end{subequations}
where $\X_t\in\Re^d$ is the state at time $t$, 
$Z_t \in\Re^m$ is the
observation process, $A$, $C$ are matrices of appropriate dimension, and
$\{B_t\}$, $\{W_t\}$ are two mutually independent Wiener
processes taking values in $\Re^d$ and $\Re^m$, respectively. 
Without loss of generality, it is assumed that the covariance
matrices associated with $\{B_t\}$, $\{W_t\}$ are identity matrices. 
The initial condition is assumed to have a Gaussian distribution, $\NN(\hat{X}_0,\Sigma_0)$, with $\Sigma_0\succ0$. The initial condition and noise processes are also independent.
The filtering problem is to compute the posterior distribution $\PP(\X_t|\clZ_t)$, where $\clZ_t=\sigma(Z_s; \leq s \leq t)$.

\newP{Kalman filter} In the linear Gaussian case, the posterior distribution $\PP(\X_t|\clZ_t)$ is Gaussian $\NN(\hat{\X}_t,\Sigma_t)$, whose mean and variance are given by the Kalman-Bucy filter,
\begin{subequations}
\begin{align}
\ud \hat{X}_t &= A \hat{X}_t \ud t + \k_t(\ud Z_t - C\hat{X}_t \ud t),
\label{eq:kalman-mean} \\
\ddt \Sigma_t &= A \Sigma_t + \Sigma_t A^T + I - \Sigma_tC^TC\Sigma_t,
\label{eq:kalman-variance}
\end{align}
\end{subequations}
where $\k_t :=\Sigma_tC^T$ is the Kalman gain.

\newP{Feedback particle filter} The feedback particle filter for linear Gaussian problem \eqref{eqn:Signal_Process}-\eqref{eqn:Obs_Process} is given by, 
\begin{equation}
\begin{aligned}
\ud \Xii_t= A\Xii_t \ud t + \ud \tilde{B}_t + \kii_t
 \big( \ud Z_t - \frac{C \Xii_t + C \hatXii_t}{2} \ud t \big), 
\end{aligned}
\label{eq:FPF-lin}
\end{equation}
where $\kii_t := \Sigmaii_tC^T$ is the Kalman gain, $\{\tilde{B}_t\}$ is a standard Wiener process, $\hatXii_t := \E[ \Xii_t |\clZ_t]$, $\Sigmaii_t
:= \E[ (\Xii_t - \hatXii_t)(\Xii_t - \hatXii_t)^T |\clZ_t]$, 
and $\Xii_0\sim\NN(\hat{X}_0,\Sigma_0)$.

\medskip
\begin{remark}
In practice, $N$ realizations of $S_t$ are simulated according to finite-$N$ system: 
\begin{equation*}
\begin{aligned}
\ud \Xii^i_t= A\Xii^i_t \ud t + \ud \tilde{B}^i_t + \kii^{(N)}_t
 \big( \ud Z_t - \frac{C \Xii^i_t + C \hatXii^{(N)}_t}{2} \ud t \big), 
\end{aligned}
\label{eq:finte-N}
\end{equation*}
for $i=1,\ldots,N$, where $\{B^i_t\}$ are independent Wiener processes, $\Xii_0^i \overset{i.i.d}{\sim} \NN(\hat{\X}_0,\Sigma_0)$, $\kii_t^{(N)}:=\Sigmaii_t^{(N)}C^T$, and
\begin{equation*}
\begin{aligned}
\hatXii_t^{(N)}&=\frac{1}{N}\sum_{i=1}^N \Xii^i_t,\quad 
\Sigmaii_t^{(N)} = \frac{1}{N}\sum_{i=1}^N (\Xii_t^i-\hatXii_t)(\Xii_t^i-\hatXii_t)^T.
\end{aligned}
\end{equation*}
The sde \eqref{eq:FPF-lin} represents the mean-field limit of the finite-$N$ algorithm, which is the focus of present work. 
\label{remark:finite-N}
\end{remark}
\medskip
\begin{theorem}
\label{thm_lin} {\bf (Exactness of linear FPF)} Consider the linear Gaussian filtering problem~\eqref{eqn:Signal_Process}-\eqref{eqn:Obs_Process}, and the linear
FPF~\eqref{eq:FPF-lin}. If $\PP(X_0)=\PP(\Xii_0)$, then
\begin{equation}
\PP(\Xii_t|\clZ_t)=\PP(\X_t | \clZ_t),\quad \forall t>0.
\label{eq:exactness-lin}
\end{equation}

\label{thm:consistency-FPF-lin}
\end{theorem}
\begin{proof}
The conditional distribution of $\Xii_t$ is a Gaussian because the sde \eqref{eq:FPF-lin} is linear and the initial condition is Gaussian. So, to prove \eqref{eq:exactness-lin}, it is sufficient to show that the conditional mean and variance of $\Xii_t$ evolve according to Kalman filtering equations. 
The equation for the conditional mean is obtained by taking conditional
mean for both sides of~\eqref{eq:FPF-lin}, using
linearity and the fact that $B_t^i$ is zero mean. 
\begin{equation}
\ud \hatXii_t = A \hatXii_t \ud t + \kii_t(\ud Z_t - C\hatXii_t\ud t).
\label{eq:Xii-mean}
\end{equation}
This equation is similar to \eqref{eq:kalman-mean} except that the gain $\k_t =\Sigma_tC^T$ \eqref{eq:kalman-mean} is replaced by $\kii_t=\Sigmaii_tC^T$. In order to show the variances are equal, define $E_t$ according to,
\begin{equation*}
E_t := \Xii_t - \hatXii_t. \label{eq:decompose}
\end{equation*}
The equation for $E_t$ is obtained by simply
subtracting~\eqref{eq:Xii-mean} from~\eqref{eq:FPF-lin}.
This gives,
\begin{equation*}
\ud E_t = (A - \frac{\Sigmaii_t C^TC}{2}) E_t + \ud \tilde{B}_t. \label{eq:E}
\end{equation*}
The equation for the variance of $E_t$ is now given by the Lyapunov equation,
\begin{equation*}
\ddt \Sigmaii_t = A \Sigmaii_t + \Sigmaii_tA^T + I - \Sigmaii_tC^TC\Sigmaii_t.
\label{eq:Xii-variance}
\end{equation*}
which is identical to \eqref{eq:kalman-variance}, starting with the same initial condition. 
Hence $\Sigmaii_t =\Sigma_t$ for all $t\geq0$. 
This also implies $\kii_t=\k_t$, which further implies that the sde for conditional mean \eqref{eq:Xii-mean} is identical to \eqref{eq:kalman-mean}. 
So $\hat{X}_t = \hatXii_t$. 
\end{proof}

\medskip
\begin{remark}{\bf (Non-uniqueness)} 
The exactness proof above suggests a general procedure to construct an exact $\Xii_t$ process. 
In particular, express $\Xii_t$ as sum of two terms, \[\Xii_t = \hatXii_t + E_t.\] 
Let $\hatXii_t$ evolve according to \eqref{eq:kalman-mean}. 
And let $E_t$ evolve according to, 
\begin{equation*}
\ud E_t = G_t E_t \ud t, \label{eq:EG}
\end{equation*} 
where $G_t$ is any solution to the matrix equation,
\begin{equation}
G_t \Sigmaii_t + \Sigmaii_t G_t^T = A \Sigmaii_t + \Sigmaii_tA^T + I - \Sigmaii_tC^TC\Sigmaii_t.
\label{eq:G-gen}
\end{equation}
By construction, the equation for the variance is given by \eqref{eq:kalman-variance}.
In general there are infinitely many solutions for \eqref{eq:G-gen}. 
These solutions are given by,
\begin{equation*}
G_t = A + \frac{1}{2}\Sigmaii_t^{-1} - \frac{1}{2}\Sigmaii_tC^TC + \Om_t\Sigmaii_t^{-1} ,
\label{eq:G-gen-N}
\end{equation*} 
where $\Om_t$ is any skew symmetric matrix. 
The choice $\Om_t = 0$ corresponds to the linear FPF in \eqref{eq:FPF-lin}.  
In section \ref{sec:opt-transp-sde}, it is shown that a symmetric choice of $G_t$ is optimal in the optimal transportation sense.    
\label{remark:uniqueness}
\end{remark}
%

\section{Background on Optimal transportation}
\label{sec:opt-transp}

\newP{Problem statement} Let $\PP_X$ and $\PP_Y$ be two given probability measures on $\Re^d$ with finite second moments.
The optimal transportation problem is to minimize
\begin{equation}
\begin{aligned}
\min_{T}~ \expect[(T(X)-X)^2], 
\end{aligned}
\label{eq:opt-transp}
\end{equation}   
over all maps $T:\Re^d \to \Re^d$ such that 
\begin{equation}
X\sim \PP_X, ~T(X) \sim \PP_Y.
\label{eq:constraint}
\end{equation}
If it exists, the minimizer $T^*$, is called the optimal transport map between $P_X$ and $P_Y$. 
The optimal cost is referred to as $L^2$-Wasserstein distance between $P_X$ and $P_Y$.
\medskip
\begin{theorem}
{{\bf (Existence and uniqueness}, Theorem.~2.12 in~\cite{villani2003}})  Consider the optimization problem \eqref{eq:opt-transp}, with constraint \eqref{eq:constraint}.
If $P_X$ is absolutely continuous with respect to the Lebesgue measure, then there exists a unique optimal map.  
\end{theorem}

\medskip
\begin{proposition}{\bf (Optimal map between Gaussians}, Prop.~7 in~\cite{givens84}) 
Consider the optimal transportation problem \eqref{eq:opt-transp}. 
Suppose $\PP_X$ and $\PP_Y$ are Gaussian distributions, $\NN(\hat{X},\Sigma_X)$ and $\NN(\hat{Y},\Sigma_Y)$, with $\Sigma_X,\Sigma_Y\succ 0$.
Then the optimal transport map between $P_X$ and $P_Y$ is given by,
\begin{equation}
T(x) = \hat{Y} + F(x-\hat{X}),
\label{eq:optmapgauss}
\end{equation}
where 
\begin{equation}
F = \Sigma_Y^{\half}(\Sigma_Y^{\half}\Sigma_X\Sigma_Y^\half)^{-\half}\Sigma_Y^\half.
\label{eq:F}
\end{equation} 
In the scalar case ($d=1$), the optimal map is,
\begin{equation*}
T(x) = \hat{Y} + \sqrt{\frac{\Sigma_{Y}}{\Sigma_X}}(x-\hat{X}).
\label{eq:OMGR}
\end{equation*}
\label{prop:opt-map-Gaussians}
\end{proposition}

\medskip
\begin{remark}
Consider affine maps of the form,
\[T(x) = \hat{Y} + F(x-\hat{X}).\] 
The constraint $T(X) \sim Y$ is satisfied for all matrices $F$ such that
\begin{equation}
F\Sigma_XF^T = \Sigma_Y.
\label{eq:F-constraint-1}
\end{equation} 
The general solution of \eqref{eq:F-constraint-1} is,
\begin{equation*}
F = \Sigma_Y^\half U \Sigma_X^{-\half},
\end{equation*}
for all orthogonal matrices $U$. 
The optimal choice of $U$ is the one that serves to make $F$ a symmetric matrix.
\end{remark}

\section{Optimal transport sde}
\label{sec:opt-transp-sde}
Consider the linear Gaussian filtering problem \eqref{eqn:Signal_Process}-\eqref{eqn:Obs_Process}, 
over a finite time interval $[0,T]$. 
The objective of this section is to design a unique optimal sde whose solution, denoted as $\Xii_t$ has conditional probability distribution equal to the conditional probability distribution of $\X_t$. That is, 
\begin{equation*}
\PP(\Xii_t | \clZ_t) = \PP(\X_t | \clZ_t),\quad \text{for}\quad t \in [0,T].
\end{equation*} 
Recall that $\PP(\X_t|\clZ_t)$ is Gaussian with mean and variance that evolve according to 
the Kalman filtering equations \eqref{eq:kalman-mean}-\eqref{eq:kalman-variance}.
\subsection{Time stepping optimization procedure}
The following time stepping optimization procedure is proposed to obtain the desired sde: 
\begin{enumerate}
\item Divide the time interval $[0,T]$ into $n \in \mathbb{N}$ equal time steps with the time instants $t_0=0 < t_1 < \ldots < t_n = T$. 
\medskip
\item Initialize a discrete time random process $\{\Xii_{t_k}\}_{k=1}^n$ according to the initial distribution (prior) of $X_0$,
\begin{equation*}
\Xii_{t_0} \sim \PP(X_0).
\end{equation*}
\item For each time step $[t_k,t_{k+1}]$, evolve the process $\Xii_{t_k}$ according to the map $T_k$,
\begin{equation}
\Xii_{t_{\ii+1}}=T_\ii(\Xii_{t_\ii}),\quad \text{for}\quad \ii=0,\ldots,n-1,
\label{eq:updatestep}
\end{equation}
where the map $T_k$ is the optimal transport map between two probability measures, $\PP(X_{t_k}|\clZ_{t_k})$ and $\PP(X_{t_{k+1}}|\clZ_{t_{k+1}})$. 

\medskip
\item Take the limit as $n\to \infty$ to obtain the continuous time process $\Xii_t$ and the desired sde,
\begin{equation*}
\ud \Xii_t = u_t^*(\Xii_t)\ud t + \k_t(\Xii_t)\ud Z_t.
\label{eq:opt-sde-uk}
\end{equation*} 
\end{enumerate}

\medskip
The procedure leads to the control laws $u^*_t$ and $\k^*_t$ that depends upon $\PP(\X_t|\clZ_t)$. 
Since $\PP(X_t|\clZ_t)$ is unknown, one can replace it with $\PP(\Xii_t|\clZ_t)$, as the two are identical by construction. 
In practice, one would also need to approximate $\PP(\Xii_t|\clZ_t)$ by simulating $N$ realizations of the process $\Xii_t$ (see Remark \ref{remark:finite-N}). 
  
The resulted sde \eqref{eq:opt-sde-uk} is referred to as the {\it optimal transport sde}.
Next, the optimal transport sde is obtained for two cases:
The scalar case in section \ref{sec:scalar}, and the vector case in section \ref{sec:vector}.
\begin{figure*}[t]
\begin{tabular}{cc}
\subfigure[]{
\includegraphics[width=0.9\columnwidth]{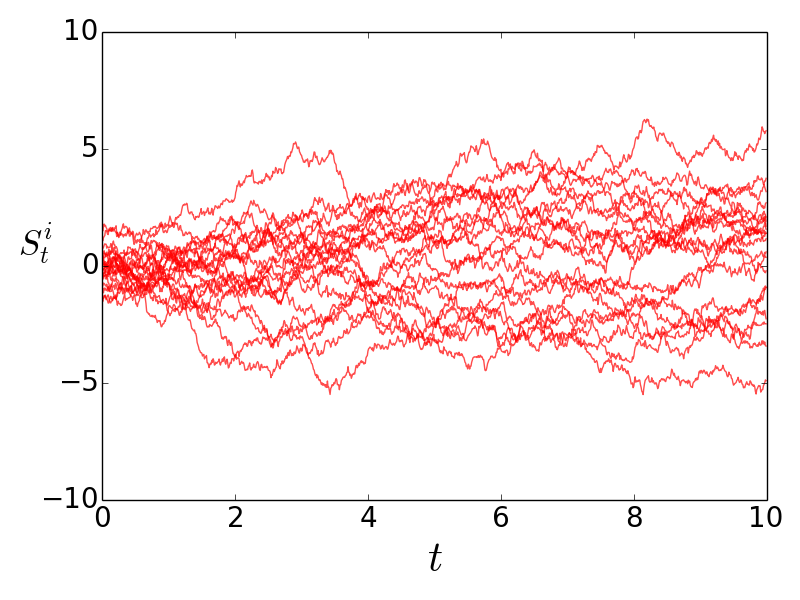}
\label{fig:traj-FPF}
}&
\subfigure[]{
\includegraphics[width=0.9\columnwidth]{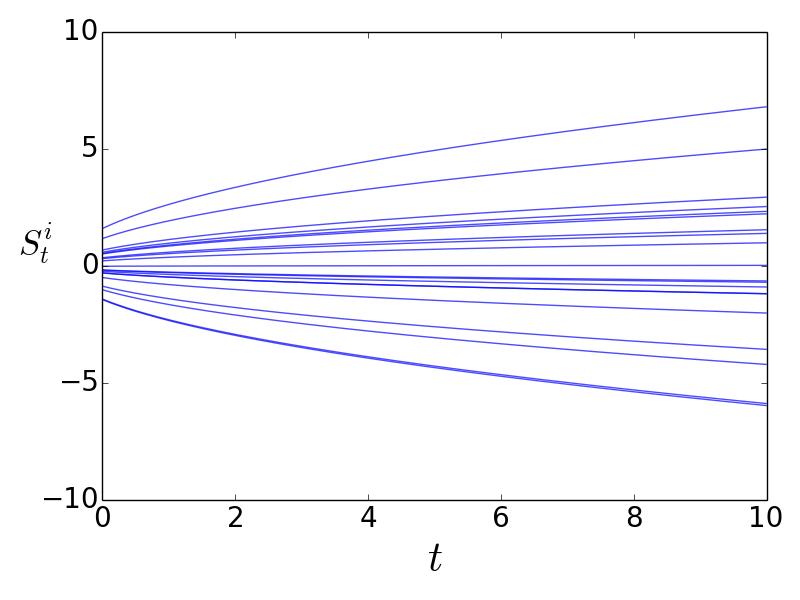}
\label{fig:traj-opt-sde}
}
\end{tabular}
\caption{Particle trajectories obtained by simulating the two models: (a) Monte-Carlo sde \eqref{eq:MC},
(b) Optimal transport sde \eqref{eq:opt-transp-MC}.}
\label{fig:traj}
\end{figure*}
\subsection{Scalar Case}
\label{sec:scalar}

This section considers the linear Gaussian filtering problem 
\eqref{eqn:Signal_Process}-\eqref{eqn:Obs_Process} in the scalar case, where $\X_t,Z_t\in\Re$.
In this case, the matrices $A$ and $C$ are scalars, and denoted by $a$ and $c$ respectively. 
The proof of the following proposition appears in Appendix \ref{proof:opt-sde}.

\medskip
\begin{proposition}
Consider the linear Gaussian filtering problem \eqref{eqn:Signal_Process}-\eqref{eqn:Obs_Process} in the scalar case. 
Assume the pair $(a,c)$ is observable, i.e $c\neq 0$. 
The optimal transport sde for this problem is given by,
\begin{equation}
\begin{aligned}
\ud \Xii_t &= a\Xii_t \ud t +\frac{1}{2\Sigmaii_t}(\Xii_t-\hatXii_t)\ud t + \kii_t(\ud Z_t - \frac{c\Xii_t + c\hatXii_t}{2}\ud t),
\end{aligned}
\label{eq:opt-sde-scalar}
\end{equation}
where $\kii_t=\Sigmaii_tc$, $\hatXii_t=\expect [\Xii_t|\clZ_t]$, $\tilde{\Sigma}_t=\expect [(\Xii_t-\hatXii_t)^2|\clZ_t]$. The sde \eqref{eq:opt-sde-scalar} is exact, i.e,
\begin{equation*}
\PP(\Xii_t|\clZ_t) = \PP(\X_t|\clZ_t)\quad \text{for}\quad t >0,
\label{eq:consistency-scalar}
\end{equation*}
if $\PP(\Xii_0)=\PP(X_0)$.
\label{prop:linear-opt-sde-scalar} 
\end{proposition}

\medskip
\begin{remark}
The Bayesian update law in the optimal transport sde \eqref{eq:opt-sde-scalar}
$\kii_t(\ud Z_t - \frac{c\Xii_t+c\hatXii_t}{2}\ud t)$
is identical to FPF \eqref{eq:FPF-lin}. The difference between \eqref{eq:opt-sde-scalar} and \eqref{eq:FPF-lin}  
is the replacement of the stochastic term $\ud B_t$ with a deterministic term $\frac{1}{2\Sigma_t}(\Xii_t-\hatXii_t)\ud t$.
In section \ref{sec:numerics}, a numerical example is described to show that this difference serves to decrease the simulation variance.
\label{remark:scalar}
\end{remark}

\subsection{Vector Case}
\label{sec:vector}

The proof of the following Proposition appears in Appendix \ref{proof:opt-sde}.
\begin{proposition}
Consider the linear Gaussian filtering problem \eqref{eqn:Signal_Process}-\eqref{eqn:Obs_Process}. 
Assume the pair $(A,C)$ is observable. 
Then the optimal transport sde for this problem is given by,
\begin{equation}
\begin{aligned}
\ud \Xii_t &=  A\hatXii_t \ud t + \kii_t(\ud Z_t - C\hatXii_t\ud t) + G_t(\Xii_t-\hatXii_t)\ud t, 
\end{aligned}
\label{eq:opt-sde-vector}
\end{equation}
where $\kii_t:=\Sigmaii_tC^T$, $\hatXii_t=\expect [\Xii_t|\clZ_t]$, $\tilde{\Sigma}_t=\expect [(\Xii_t-\hatXii_t)^2|\clZ_t]$,
and $G_t$ is the solution of the matrix equation,
\begin{equation}
G_t\Sigmaii_t + \Sigmaii_tG_t = A\Sigmaii_t+ \Sigmaii_t A^T + I - \Sigmaii_tC^TC\Sigmaii_t
\label{eq:G}
\end{equation}
The sde \eqref{eq:opt-sde-vector} is exact, i.e,
\begin{equation*}
\PP(\Xii_t|\clZ_t) = \PP(\X_t|\clZ_t)\quad \text{for}\quad t >0,
\label{eq:consistency-vector}
\end{equation*}
if $\PP(\Xii_0)=\PP(X_0)$.
\label{prop:opt-sde-vector}
\end{proposition}

\medskip
\begin{remark}
The matrix equation \eqref{eq:G} is the Lyapunov equation with $\Sigmaii_t \succ0$. 
A unique solution $G_t$ exists and is furthermore a symmetric matrix.
Next comparisons are drawn between the optimal transport sde \eqref{eq:opt-sde-vector} and the FPF \eqref{eq:FPF-lin}. 
To aid the comparison $G_t$ is expressed as, 
\begin{equation}
G_t = A + \frac{1}{2}\Sigmaii_t^{-1} - \frac{1}{2}\Sigmaii_tC^TC + \Om_t\Sigmaii_t^{-1},
\label{eq:G-opt}
\end{equation} 
where $\Om_t$ is the unique solution to the following matrix equation,
\begin{equation}
\Om_t\Sigmaii_t^{-1} + \Sigmaii_t^{-1}\Om_t = A^T - A + \half(\Sigmaii_tC^TC - C^TC\Sigmaii_t).
\label{eq:N}
\end{equation}
Using \eqref{eq:G-opt}, the optimal transport sde \eqref{eq:opt-sde-vector} has the form,
\begin{equation*}
\begin{aligned}
\ud \Xii_t= &A\Xii_t \ud t + \frac{1}{2}\Sigmaii_t^{-1}(\Xii_t-\hatXii_t) \ud t + \kii_t
 \big( \ud Z_t - \frac{C \Xii_t + C \hatXii_t}{2} \ud t \big) +\\
 &\Om_t\Sigmaii_t^{-1}(\Xii_t-\hatXii_t) \ud t. 
\end{aligned}
\label{eq:FPF-N}
\end{equation*} 
Comparing to the original FPF \eqref{eq:FPF-lin} there are two differences.
\begin{enumerate}
\item The stochastic term $\ud B_t$ is replaced with the deterministic term, $\frac{1}{2}\Sigmaii_t^{-1}(\Xii_t-\hatXii_t)\ud t$,
which similar to scalar case.
\medskip
\item In the vector case, there is an "extra" term  
$\Om_t\Sigmaii_t^{-1}(\Xii_t-\hatXii_t)\ud t$,
where $\Om_t$ is the skew symmetric matrix, solution to \eqref{eq:N}. 
The extra term does not effect the distribution of $\Xii_t$. 
This term can be viewed as a correction term that serves to make the dynamics "symmetric" and hence optimal in the optimal transportation sense. 
This term was absent in scalar case because the only $1\times 1$ skew symmetric matrix is $0$.
\end{enumerate}

\end{remark}
\medskip
\begin{remark}
In practice, solving the Lyapunov equation \eqref{eq:G} for $G_t$ requires additional computational cost, which was absent in the original FPF. 
\end{remark}
\begin{figure*}[t]
\begin{tabular}{cc}
\subfigure[]{
\includegraphics[width=0.9\columnwidth]{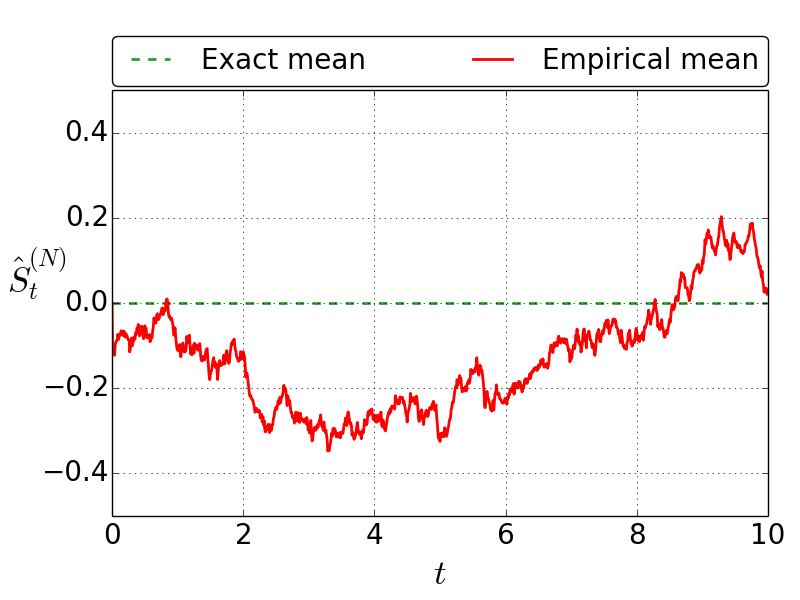}
\label{fig:mean-MC}
}&
\subfigure[]{
\includegraphics[width=0.9\columnwidth]{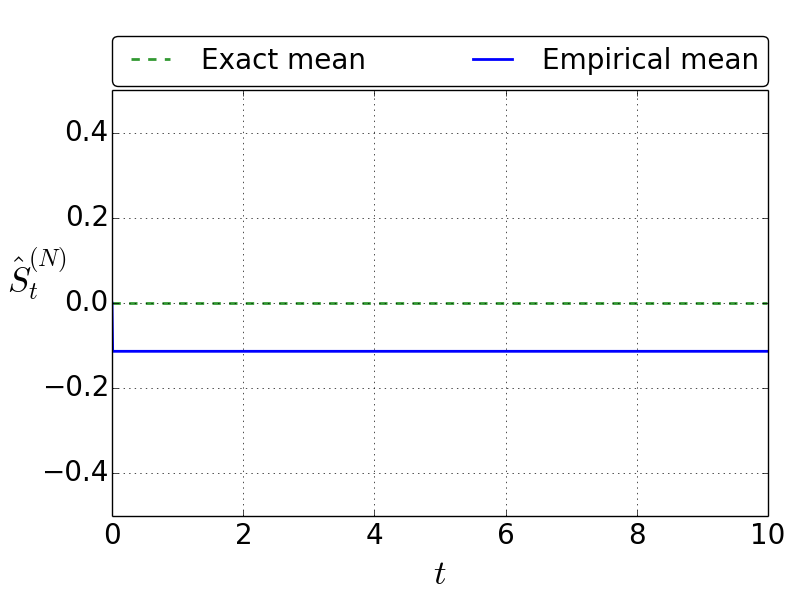}
\label{fig:mean-OT}
}
\end{tabular}
\caption{Empirical mean of the particles in a single simulation: (a) Monte-Carlo sde \eqref{eq:MC}, (b) Optimal transport sde \eqref{eq:opt-transp-MC}, where number of particles $N=80$ for each case.}
\end{figure*}
\section{Numerics}
\label{sec:numerics}
The aim of this section is to show that a deterministic approach, such as the one discussed in this paper, can reduce the simulation variance. 
This phenomenon is illustrated in the simplest possible setting of simulating a Brownian motion.    

\newP{Problem statement} Consider a real-valued stochastic process $\X_t \in \Re$,
\begin{equation}
\ud \X_t = \ud B_t,
\label{eq:brownian}
\end{equation}
where the initial condition $X_0$ is distributed according to a Gaussian distribution $\NN(0,1)$, and $\{B_t\}$ is standard Wiener process.
Indeed, the probability distribution of $\X_t$, denoted as $\PP(\X_t)$, is a Gaussian $\NN(0,1+t)$. 
The objective is to approximate $\PP(\X_t)$ using the Monte-Carlo and the optimal transport methods, and to compare the accuracy of these methods. 

\newP{Monte-Carlo} In the first method, a system of $N$ particles are simulated independently according to,
\begin{equation}
\ud \Xii^i_t = \ud B^i_t, \quad \text{for}\quad i=1,\ldots,N,
\label{eq:MC}
\end{equation}
where $\{B_t^i\}$ are independent Wiener processes, and $\Xii^i_0$ are i.i.d samples drawn from initial distribution $\NN(0,1)$. 

\newP{Optimal transport} In the second method, a system of $N$ particles are simulated according to the optimal transport sde,
\begin{equation}
\ud \Xii^i_t = \frac{1}{2\Sigmaii^{(N)}_t}(\Xii_t-\hatXii^{(N)})\ud t, 
\label{eq:opt-transp-MC}
\end{equation}  
for $i=1,\ldots,N$, where $\Xii^i_0$ are i.i.d samples drawn from initial distribution $\NN(0,1)$, and 
\begin{equation*}
\begin{aligned}
\hatXii_t^{(N)}&:=\frac{1}{N}\sum_{i=1}^N \Xii^i_t,\quad 
\Sigmaii_t^{(N)} := \frac{1}{N}\sum_{i=1}^N (\Xii_t^i-\hatXii_t^{(N)})^2.
\end{aligned}
\end{equation*}
Note that in the limit as $N \to \infty$ one obtains  
\[\ud \Xii_t=\frac{1}{2\Sigmaii_t}(\Xii_t-\hatXii_t)\ud t\]
which according to Remark \ref{remark:scalar} is the deterministic counterpart of \eqref{eq:MC}. 

\newP{Particle trajectories} The trajectories of particles obtained by simulating the Monte-Carlo model \eqref{eq:MC} and the optimal transport model \eqref{eq:opt-transp-MC} are depicted in  
Figure \ref{fig:traj-FPF} and Figure \ref{fig:traj-opt-sde}, respectively.

\newP{Estimating the mean}
The mean of $X_t$ is estimated as the empirical mean of the particles,
\begin{equation*}
\expect[\X_t] \approx \hat{\Xii}_t^{(N)} = \frac{1}{N}\sum_{i=1}^N \Xii^i_t
\end{equation*} 
Figure \ref{fig:mean-MC} and Figure \ref{fig:mean-OT} depict the result for Monte-Carlo method and Optimal transport method respectively. 
As expected, the optimal transport method, has less "randomness". 
In fact, one can obtain an explicit formula for the empirical mean in this case.
\begin{equation*}
\begin{aligned}
\ud \hatXii^{(N)}_t = \frac{1}{N}\sum_{i=1}^N \ud \Xii^i_t= &\frac{1}{N}\sum_{i=1}^N\frac{1}{2\Sigmaii_t^{(N)}}(\Xii_t^i-\hatXii_t^{(N)})\ud t =0. \\
\end{aligned}
\end{equation*}
So, the empirical mean remains constant, equal to its initial value,
\begin{equation*}
\hatXii^{(N)}_t = \hatXii^{(N)}_0 .
\label{eq:emp-mean}
\end{equation*}
\begin{figure*}[t]
\begin{tabular}{cc}
\subfigure[]{
\includegraphics[width=0.9\columnwidth]{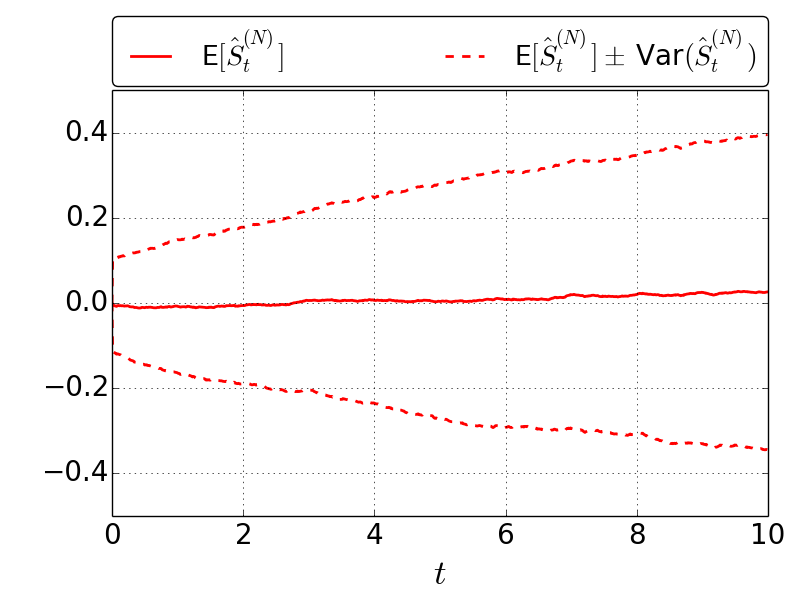}
\label{fig:sim-var-mean-MC}
}&
\subfigure[]{
\includegraphics[width=0.9\columnwidth]{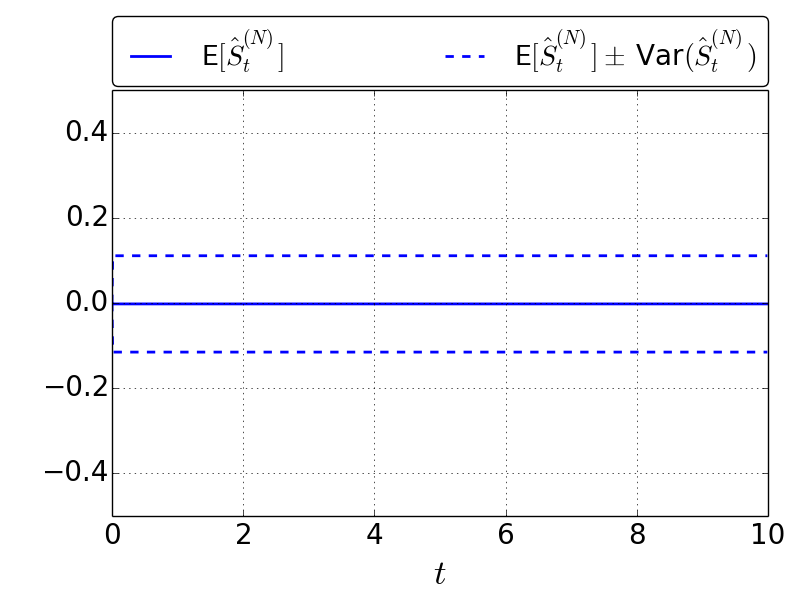}
\label{fig:sim-var-mean-OT}
}
\end{tabular}
\caption{The average of the empirical mean $\hatXii^{(N)}$ $\pm$ simulation variance, over $500$ simulations: (a) Monte-Carlo, (b) Optimal transport, where number of particles $N=80$ for each case.}
\end{figure*}

\newP{Simulation variance} 
Any estimate obtained using a simulation with finitely many particles is a random variable. 
The variance of this random variable is called simulation variance; c.f,~\cite{glynn2007}. 
Figure \ref{fig:sim-var-mean-MC} and Figure \ref{fig:sim-var-mean-OT} depict the mean and variance of $\hatXii_t^{(N)}$, obtained after averaging over $500$ simulations, for the Monte-Carlo method and the optimal transportation method respectively. 
The figure shows that the simulation variance for the Monte-Carlo method increases with time, whereas for the optimal transport method, it remains constant. 
In this simple case, this observation can also be verified analytically:
\begin{enumerate}
\item Monte-Carlo method:
%
\begin{equation*}
\begin{aligned}
\var(\hatXii^{(N)}_t) &= \var(\frac{1}{N}\sum_{i=1}^N\Xii^i_t)
=\var(\frac{1}{N}\sum_{i=1}^N(\Xii^i_0+B^i_t))=\frac{1+t}{N}.
\end{aligned}
\end{equation*}
\medskip
\item Optimal transport:
\begin{equation*}
\begin{aligned}
\var(\hatXii_t^{(N)}) &= \var(\hatXii_0^{(N)}) = \var(\frac{1}{N}\sum_{i=1}^N\Xii^i_0) = \frac{1}{N}.
\end{aligned}
\end{equation*}
\end{enumerate}
The calculation shows that for the Monte-Carlo method, the simulation variance grows with time, whereas for the optimal transport method, it remains constant.
In fact, the same result holds true for estimating the variance as shown in Figure~\ref{fig:sim-var-var-MC} and Figure~\ref{fig:sim-var-var-OT}, and explicitly given by,  
\begin{equation*}
\begin{aligned}
\text{Monte-Carlo:} &\quad \var(\Sigmaii_t^{(N)}) \approx \frac{3(1+t)^2}{N},\\
\text{Optimal transport:}&\quad \var(\Sigmaii_t^{(N)}) \approx \frac{3}{N}.
\end{aligned}
\end{equation*}

\begin{remark}
The result shows that in the optimal transport model \eqref{eq:opt-transp-MC}, $\Xii_1^t,\ldots,\Xii_t^N$ are correlated and the correlation serves to reduce the simulation variance. 
\end{remark}
\section{Conclusion}      
In this paper, an optimal transport formulation of the FPF algorithm is introduced. 
The optimal transport FPF is obtained by employing a time-step optimization procedure. 
In future, one would like to relate the current optimization procedure to an optimal control problem with an integral cost. 
The other extension would be to consider the general nonlinear non-Gaussian filtering problem.  
     
\appendix
\subsection{Proof of Prop. \ref{prop:opt-map-Gaussians}}
Let $X \sim \NN(\hat{X},\Sigma_X)$, and $Y:=T(X) \sim \NN(\hat{Y},\Sigma_Y)$. 
Consider the special case $\hat{X}=\hat{Y}=0$. 
The transportation cost is,
\begin{equation*}
\begin{aligned}
\expect\left[|Y - X|^2\right] &= \expect\left[X^TX\right] + \expect\left[Y^TY\right] -2\expect\left[X^TY\right] \\
&=\trace(\Sigma_X) + \trace(\Sigma_Y) -  2\expect\left[X^TY\right].
\end{aligned}
\end{equation*}
This shows that minimization of the transportation cost, is equivalent to maximization of the covariance $\expect[X^TY]$. 
For the covariance, one has the upper bound,
\begin{equation*}
2\expect\left[X^TY\right] \leq \trace(F\Sigma_X) + \trace(\Sigma_YF^{-1}),
\end{equation*} 
where $F=\Sigma_Y^{\half}(\Sigma_Y^{\half}\Sigma_X\Sigma_Y^\half)^{-\half}\Sigma_Y^\half$, given by \eqref{eq:F}. This bound is obtained by using the inequality,
\begin{equation*}
\begin{aligned}
0\leq&\expect\left[|F^{\half}X - F^{-\half}Y|_2^2\right]\\ 
=&\expect\left[X^TFX\right] + \expect\left[Y^TF^{-1}Y\right] - 2\expect\left[X^TY\right]\\
=&\trace(F\Sigma_X)+ \trace(\Sigma_YF^{-1}) - 2\expect\left[X^TY\right],
\end{aligned}
\end{equation*}
where symmetry of $F$ is used.
Finally the Proposition is proved by showing that the optimal map
\[Y^*=FX\]
achieves the upper bound, and therefore  minimizes the cost.
Indeed, one can easily check that $Y^* \sim \NN(\hat{Y},\Sigma_Y)$, and
\begin{equation*}
2\expect[X^TY^*] = 2\expect[X^TFX] = 2\trace(F\Sigma_X) = \trace(F\Sigma_X) + \trace(\Sigma_YF^{-1}),
\end{equation*}
where the identity $F\Sigma_XF=\Sigma_Y$ is used. 
The proof for the non-zero mean case is similar.
\qed
\begin{figure*}[t]
\begin{tabular}{cc}
\subfigure[]{
\includegraphics[width=0.9\columnwidth]{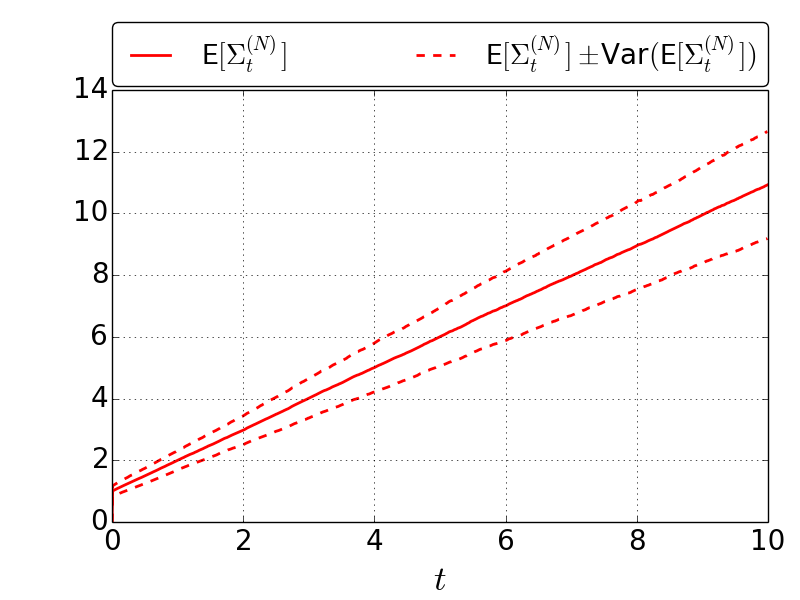}
\label{fig:sim-var-var-MC}
}&
\subfigure[]{
\includegraphics[width=0.9\columnwidth]{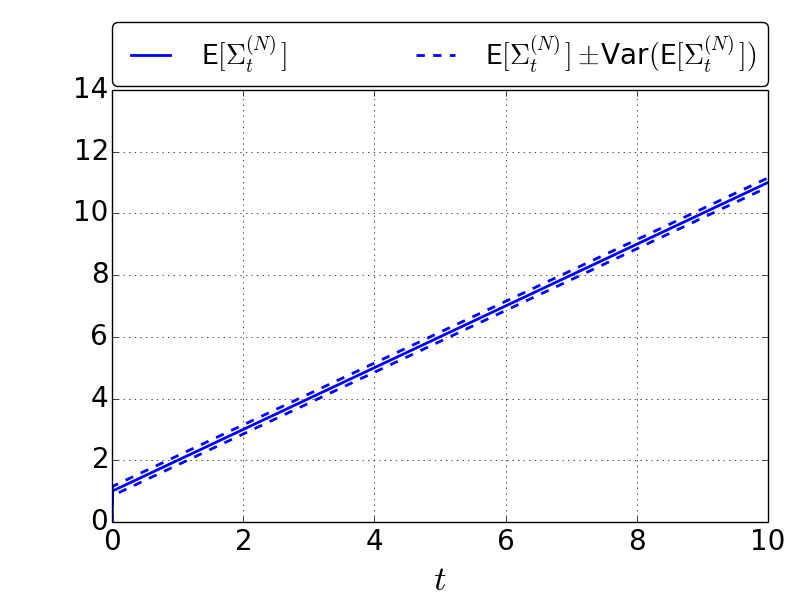}
\label{fig:sim-var-var-OT}
}
\end{tabular}
\caption{The average of the empirical variance $\Sigmaii_t^{(N)}$ $\pm$ simulation variance, over $500$ simulations for the two methods: (a) Monte-Carlo, (b) Optimal transport, where the number of particles $N=80$.}
\end{figure*}
\subsection{Proof of Prop.~\ref{prop:linear-opt-sde-scalar} and Prop.~\ref{prop:opt-sde-vector}}
\label{proof:opt-sde}
The key step in the proof is the following Lemma, 
\begin{lemma}
Consider the ode \eqref{eq:kalman-variance}. Let $\Sigma_t$ be the solution for $t\in [0,T]$. 
Then the following relationship holds ,
\begin{equation}
\Sigma_{t + \Delta t}^{\half}
(\Sigma_{t+ \Delta t}^{\half}
\Sigma_t\Sigma_{t+\Delta t}^\half)^{-\half}
\Sigma_{t+\Delta t}^\half =
I + G_t\Delta t + O(\Delta t^2),
\label{eq:approx-Sigma-vector}
\end{equation}
where $G_t$ is the solution to the matrix equation,
\begin{equation}
G_t\Sigma_t + \Sigma_tG_t = A\Sigma_t+ \Sigma_t A^T + I - \Sigma_tC^TC\Sigma_t,
\label{eq:G-lemma}
\end{equation}
and the second order term is uniformly bounded for all $t \in [0,T]$.
\label{lemma:approx-Sigma-vector}
\end{lemma}
\begin{proof}
The solution $\Sigma_t$ is positive and bounded since the system is observable~\cite{ocone1996}. 
Fix $t \in [0,T]$, and define 
\begin{equation*}
F(s):=\Sigma_{t + s}^{\half}
(\Sigma_{t+ s}^{\half}
\Sigma_t\Sigma_{t+s}^\half)^{-\half}
\Sigma_{t+s}^\half. 
\end{equation*}
The relationship \eqref{eq:approx-Sigma-vector} is obtained by considering the Taylor series of $F(s)$ at $s=0$,
\begin{equation*}
F(\Delta t) = I + \dot{F}(0)\Delta t + \frac{1}{2}\ddot{F}(\tau)\Delta t^2,
\end{equation*} 
for some $\tau \in [0,\Delta t]$, and showing that $\dot{F}(0)=G_t$.
This is verified by considering,
\begin{equation*}
F(s)\Sigma_tF(s)= \Sigma_{t+s}.
\end{equation*}  
On evaluating the derivative with respect to $s$ at $s=0$,
\begin{equation*}
\dot{F}(0)\Sigma_t + \Sigma_t\dot{F}(0) = A\Sigma_t+ \Sigma_t A^T + I - \Sigma_tC^TC\Sigma_t.
\end{equation*}
Since the solution to the Lyapunov equation \eqref{eq:G-lemma} is unique, $\dot{F}(0)=G_t$.
Also the second order derivative is uniformly bounded for all $t \in [0,T]$, by the observability assumption. 
\end{proof}

\begin{proof}(Prop.~\ref{prop:opt-sde-vector})
The proof of exactness is similar to the proof of Theorem~\ref{thm:consistency-FPF-lin} and is omitted. 
In order to obtain the optimal transport sde, the time stepping procedure is used.
The key step in the procedure is to obtain the optimal transport map $T_k$.
The optimal map is between two Gaussians, $\NN(\hat{X}_{t_\ii},\Sigma_{t_\ii})$ and $\NN(\hat{X}_{t_{\ii+1}},\Sigma_{t_{\ii+1}})$.
By Proposition~\ref{prop:opt-map-Gaussians}, the optimal map is,
\begin{equation*}
\Xii_{t_{\ii+1}}=\hat{X}_{t_{\ii+1}} + F_k(\Xii_{t_\ii}-\hat{X}_{t_{\ii}}),
\label{eq:XiiTempRd}
\end{equation*} 
where $F_\ii=\Sigma_{t_{\ii+1}}^{\half}(\Sigma_{t_{\ii+1}}^{\half}\Sigma_{t_{\ii}}\Sigma_{t_{\ii+1}}^\half)^{-\half}\Sigma_{t_{\ii+1}}^\half$.
Using Lemma \ref{lemma:approx-Sigma-vector}, 
\begin{equation*}
\begin{aligned}
\Xii_{t_{\ii+1}}
&=\hat{X}_{t_{\ii+1}} + (\Xii_{t_\ii} - \hat{X}_{t_\ii}) +  G_{\ii}(\Xii_\ii - \hat{X}_\ii)\Delta t + O(\Delta t^2). 
\end{aligned}
\end{equation*}
To obtain the sde, take a sum over $k=0,1,\ldots,n-1$, 
\begin{equation*}
\begin{aligned}
\Xii_{t_n}&=\Xii_{t_0} + \hat{X}_{t_n}-\hat{X}_{t_0} + \sum_{\ii=0}^{n-1}\big[G_\ii(\Xii_{t_\ii} - \hat{X}_{t_\ii})\Delta t + O(\Delta t^2)\big].
\end{aligned}
\end{equation*}
In the limit as $\Delta t \to 0$,
\begin{equation*}
\begin{aligned}
\Xii_{t_n}&= \Xii_{t_0} + \hat{X}_{t_n}-\hat{X}_{t_0}+\int_{0}^{t}G_s(\Xii_s-\hat{X}_s)\ud s.
\end{aligned}
\end{equation*}
where the uniform boundedness of the second order term is used. 
The associated sde is,
\begin{equation*}
\ud \Xii_t = \ud \hat{X}_t + G_t(\Xii_t - \hat{X}_t)\ud t,
\end{equation*}
where $\ud \hat{X}_t$ is given by \eqref{eq:kalman-mean}. 
Finally one obtains \eqref{eq:opt-sde-vector} by replacing $\hat{X}_t$ and $\Sigma_t$ with $\hatXii_t$ and $\Sigmaii_t$ respectively, which are identical by exactness. 
\end{proof}
\medskip
Proof of Prop~\ref{prop:linear-opt-sde-scalar} is a special case of this proof. 
\bibliographystyle{plain}
\bibliography{fpfbib}
\end{document}